\documentclass[12pt]{iopart}

\usepackage{iopams}

\usepackage{amsthm}

\usepackage{mathrsfs}
\usepackage{amssymb}

\newtheorem{theorem}{Theorem}
\newtheorem{lemma}[theorem]{Lemma}

\begin{document}

\note[Semi-symmetric spacetimes]{Note on (conformally) semi-symmetric spacetimes}

\author{Ingemar Eriksson$^1$ and Jos\'e M M Senovilla$^2$}

\address{$^1$ Google Switzerland GmbH, Brandschenkestrasse 110, CH-8002 Z\"urich, Switzerland}
\address{$^2$ F\'{\i}sica Te\'orica, Universidad del Pa\'{\i}s Vasco, Apartado 644, 48080 Bilbao, Spain}
\eads{\mailto{ineri3@gmail.com}, \mailto{josemm.senovilla@ehu.es}}
\begin{abstract}
We provide a simple proof that conformally semi-symmetric spacetimes are actually semi-symmetric. We also present a complete refined classification of the semi-symmetric spacetimes.
\end{abstract}


\section{Introduction}
Semi-symmetric spaces were introduced by Cartan \cite{C1}. They are characterized by the curvature condition
\begin{equation} \label{SemiSymmetric}
    \nabla_{[a}\nabla_{b]}R_{cdef} = 0,
\end{equation}
where $R_{cdef}$ denotes the Riemann tensor and round and square brackets enclosing indices indicate symmetrization and antisymmetrization, respectively. Geometrically, they satisfy the property that the sectional curvature relative to any tangent plane at any point remains invariant after parallel translation along an infinitesimal parallelogram \cite{HV,HV2}.

Semi-symmetric proper Riemannian manifolds were studied in \cite{Sz,Sz1}, and they are considered the natural generalization of locally symmetric spaces, i.e., those satisfying $\nabla_{b}R_{cdef} = 0$. This is not the case for general semi-Riemannian manifolds, where the intermediate condition
\begin{equation}
\nabla_{a}\nabla_{b}R_{cdef} = 0 \label{2sym}
\end{equation}
is actually feasible \cite{S2007,S2008}. Semi-Riemannian manifolds satisfying (\ref{2sym})  are called second-order symmetric, and are obviously semi-symmetric, but the converse is not true in general. Semi-symmetric spacetimes (4-dimensional Lorentzian manifolds) have been considered in \cite{HV} as a particular case of the more general pseudo-symmetric case \cite{DDVV}.

A semi-Riemannian manifold is said to be {\em conformally semi-symmetric} if the Weyl tensor $C_{abcd}$ satisfies
\begin{equation} \label{ConformallySemiSymmetric}
    \nabla_{[a}\nabla_{b]}C_{cdef} = 0;
\end{equation}
and {\em Ricci semi-symmetric} if the Ricci tensor $R_{ab}\equiv R^c{}_{acb}$ satisfies
\begin{equation} \label{RicciSemiSymmetric}
    \nabla_{[a}\nabla_{b]}R_{cd} = 0.
\end{equation}
It is obvious that semi-symmetric spaces are automatically both conformally and Ricci semi-symmetric. However, none of these two by itself should imply the former in general. Surprisingly, for dimensions greater than four conformal semi-symmetry implies semi-symmetry (for non-conformally flat spaces), see \cite{G,DG}.

However, in the four-dimensional proper Riemannian case there are examples of conformally semi-symmetric spaces which are not semi-symmetric (see Lemma 1.1 in \cite{D}). Thus, the question arises of what happens for other signatures in four dimensions. 
In this short note we want to consider the case of Lorentzian signature, providing a very simple proof the non-trivial result that  conformally semi-symmetric spacetimes (with non-zero Weyl tensor) are automatically Ricci semi-symmetric. Therefore, semi-symmetry and conformal semi-symmetry are essentially equivalent in this case. This equivalence is implicit in \cite{HV3} (see Theorem 95), but we would like to present herein a very simple and completely direct proof of the result. 

\section{Main results}
Even though the calculations can be carried out in tensor formalism, it is much more efficient to resort to spinorial techniques. Thus, we will use the standard nomenclature and conventions in \cite{Penrose1984, Penrose1986}, except for the scalar curvature $R\equiv R^c{}_c$, as we will not use the usual $\Lambda \equiv R/24$ notation.

\begin{lemma}
Conformally semi-symmetric spacetimes are of Petrov type \textbf{N}, with
\begin{equation} \label{ConditionN}
  \Psi_{ABCD} = \Psi_4 o_A o_B o_C o_D, \qquad
  \Phi_{ABA'B'} = \Phi_{22} o_A o_B o_{A'} o_{B'}, \qquad
  R = 0,
\end{equation}
or Petrov type \textbf{D}, with
\begin{equation} \label{ConditionD}
  \Psi_{ABCD} = 6 \Psi_2 o_{(A} o_B \iota_C \iota_{D)}, \quad
  \Phi_{ABA'B'} = 4 \Phi_{11} o_{(A} \iota_{B)} o_{(A'} \iota _{B')}, \quad
  R = -12 \Psi_2.
\end{equation}
\end{lemma}

\begin{proof}
In spinors (\ref{ConformallySemiSymmetric}) is equivalent to
$\Box_{AB} \Psi_{CDEF} = 0$ and $\Box_{A'B'}\Psi_{CDEF} = 0$, or
\begin{eqnarray}
     {\rm X}_{AB(C}{}^G \Psi_{DEF)G} &= 0, \label{WeylCondition1} \\
      \Phi_{A'B'(C}{}^G \Psi_{DEF)G} &= 0. \label{WeylCondition2}
\end{eqnarray}
The first condition, (\ref{WeylCondition1}), can be rewritten as
\begin{equation}
   24  \Psi_{AB(C}{}^G \Psi_{DEF)G} = -R \left( \varepsilon_{A(C}\Psi_{DEF)B} +
                                            \varepsilon_{B(C}\Psi_{DEF)A} \right).
\end{equation}
Contracting this over $BC$ yields
\begin{equation}
    12 \Psi_{(AD}{}^{BG} \Psi_{EF)BG} = R\,  \Psi_{ADEF}.
\end{equation}
This can only be satisfied for Petrov types \textbf{D} and \textbf{N}\footnote{The reason is that, for a general $\Psi_{ABCD} = \alpha_{(A} \beta_B \gamma_C \delta_{D)}$, one gets non-vanishing terms like $\alpha_A \alpha_C \beta_B \beta_D \left(\gamma^E \delta_E \right)^2$ on the left-hand side which are not present in $\Psi_{ABCD}$.}.
For type \textbf{N},
$\Psi_{ABCD} = \Psi_4 o_A o_B o_C o_D$ and $R = 0$, while for type \textbf{D} we have
$\Psi_{ABCD} = 6\Psi_2 o_{(A}o_B\iota_C\iota_{D)}$ and $R = -12 \Psi_2$, see \cite{Penrose1986} p.261.
Condition (\ref{WeylCondition1}) can then be verified to be
satisfied in both cases.

For the second condition, (\ref{WeylCondition2}), in type \textbf{N} one gets $\Phi_{i0} = 0$ and
$\Phi_{i1} = 0$, for all $i=0,1,2$, after contracting with $o^C$ and $\iota^C \iota^D \iota^E \iota ^F$, respectively. Hence,
\begin{equation}
    \Phi_{ABA'B'} = \Phi_{22} o_A o_B o_{A'} o_{B'}.
\end{equation}
Similarly, for type \textbf{D}, contracting with $o^C o^D o^E$ and $\iota^A \iota^B \iota^C$ gives
\begin{equation}
    \Phi_{ABA'B'} = 4\Phi_{11}o_{(A}\iota_{B)}o_{(A'}\iota_{B')},
\end{equation}
which can then be checked to satisfy (\ref{WeylCondition2}).
\end{proof}

\begin{lemma}
Conformally semi-symmetric spacetimes are Ricci semi-symmetric.
\end{lemma}

\begin{proof}
In spinors (\ref{RicciSemiSymmetric}) is equivalent to
$\varepsilon_{A'B'}\Box_{AB} \Phi_{CDC'D'} + \varepsilon_{AB}\Box_{A'B'} \Phi_{CDC'D'}= 0$. We show that
$\Box_{AB} \Phi_{CDC'D'} = 0$, or equivalently that
\begin{equation}
    {\rm X}_{ABC}{}^E \Phi_{EDC'D'} +  {\rm X}_{ABD}{}^E \Phi_{CEC'D'} +
    \Phi_{ABC'}{}^{E'} \Phi_{CDE'D'} + \Phi_{ABD'}{}^{E'} \Phi_{CDC'E'}  = 0.
\end{equation}

The two last terms vanish here, since in both cases, (\ref{ConditionN}) and (\ref{ConditionD}),
$\Phi_{ABA'B'}$ is of the form $\Phi_{AB}\bar\Phi_{A'B'}$, and expanding the other two terms gives
\begin{eqnarray*}
 & 24  \Psi_{ABC}{}^E \Phi_{EDC'D'} + 24 \Psi_{ABD}{}^E \Phi_{CEC'D'} \\
  +&R
          \left( \varepsilon_{AC}\Phi_{BDC'D'} + \varepsilon_{BC}\Phi_{ADC'D'}
                +\varepsilon_{AD}\Phi_{BCC'D'} + \varepsilon_{BD}\Phi_{ACC'D'} \right) = 0.
\end{eqnarray*}
This is trivially satisfied for type \textbf{N}, while substitution of (\ref{ConditionD}) shows that it is also satisfied for type \textbf{D}.
\end{proof}

The above two lemmas imply our main result:
\begin{theorem}
In four-dimensional ---non-conformally flat--- spacetimes, conformal semi-symmetry is equivalent to semi-symmetry,
\begin{equation}
  \nabla_{[a}\nabla_{b]}C_{cdef} = 0 \qquad \Longleftrightarrow \qquad \nabla_{[a}\nabla_{b]}R_{cdef} = 0 .
\end{equation}
Furthermore, the semi-symmetric spacetimes are of Petrov types \textbf{D}, \textbf{N}, or \textbf{O}.
\end{theorem}

\section{Classification}
A classification of semi-symmetric spacetimes was given in \cite{HV} in the context of pseudo-symmetric spacetimes. However, this classification only considered the algebraic restrictions arising from the condition (\ref{SemiSymmetric}), without analyzing the compatibility differential conditions derived thereof.

Here, we present a complete classification of the (conformally) semi-symmetric spacetimes. From the previous results, there are two main possibilities, types \textbf{D} and \textbf{N}. Then we have:
\subsection{Type \textbf{D}}
Taking into account the previous results the Ricci tensor can be written as
$$
R_{ab}=A k_{(a}\ell_{b)} + B m_{(a}\bar{m}_{b)},
$$
where $k_a =o_A o_{A'}$ and $\ell_a =\iota_A \iota_{A'}$ are the two multiple principal null directions and $m_a=o_A\iota_{A'}$ is a complex null vector completing the usual null tetrad. The scalars $A$ and $B$ are proportional to $3\Psi_2\pm 2\Phi_{11}$, respectively.

Observe that $\Psi_2$ is real, so that these spacetimes are purely electric \cite{Exact} with respect to the timelike direction defined by $k_a+\ell_a$.

The Bianchi identities written in the Newman-Penrose (NP) formalism \cite{NP,Exact} provide easily the following necessary conditions
\begin{eqnarray}
A\sigma = A \lambda = A \mu =A \rho =0, \label{A}\\
B \kappa = B \nu = B \pi = B \tau =0 \, .
\end{eqnarray}
The {\em generic case} is given by $A\neq 0 \neq B$, which leads to $\sigma = \lambda = \mu = \rho = \kappa = \nu = \pi = \tau =0$. This immediately implies that $\nabla_a k_b = v_a k_b$ and $\nabla_a\ell_b= -v_a \ell_b$ for some vector field $v_a$. In other words, the two principal null directions are recurrent. It follows that the tensor $k_a \ell_b$ is covariantly constant and therefore that the spacetime is 2 $\times$ 2 decomposable, see e.g.  \cite{S2008,Exact}. Thus, the generic type-\textbf{D} semi-symmetric spacetimes are precisely the 2 $\times$ 2 decomposable ones.

The special cases are given by $3\Psi_2\pm 2\Phi_{11}=0$.

If $2\Phi_{11} = - 3\Psi_2$ ($A=0$) we get $\kappa=\nu=\tau=\pi = 0$. The principal null directions are geodesic but not necessarily shear-free. One can then check, with some patience,  that all remaining NP equations are compatible.

If $2\Phi_{11} = 3\Psi_2$ ($B=0$) we get $\rho = \mu= \sigma = \lambda= 0$, so that in general the principal null directions are not necessarily geodesic. Again all remaining NP equations are compatible. (In this case it is worth mentioning that, as one can easily see, the Einstein tensor $G_{ab}= -Am_{(a}\bar{m}_{b)}$ can never satisfy the dominant energy condition.)

Note finally that 1 $\times$ 3 decomposable spacetimes which are the product of a real line times a semi-symmetric 3-dimensional Riemannian manifold belong to these cases, because all static spacetimes are necessarily of Petrov types \textbf{I}, \textbf{D}, or \textbf{0}.

\subsection{Type \textbf{N}}
In this case, the Ricci tensor takes the null radiation form
$$
R_{ab} =2 \Phi_{22} k_a k_b \, .
$$
The Bianchi identities in the NP form provide the following conditions
\begin{eqnarray*}
\kappa =0, \qquad \sigma \Psi_4 =\rho \Phi_{22}
\end{eqnarray*}
together with other differential relations.  Thus, in general the unique principal null direction $k^a$ is geodesic but can be shearing, twisting and expanding. The remaining Bianchi identities and the rest of the NP equations can be seen to be compatible.

The second-order symmetric spacetimes ---characterized by (\ref{2sym})--- which are not locally symmetric must have a covariantly constant null vector field \cite{S2008}, from where one can infer that they belong to the particular class of  type-\textbf{N} semi-symmetric spacetimes with $\sigma =\rho =0$, so that $\Psi_4$ and $\Phi_{22}$ are in that case independent.

\section*{Acknowledgements}
We are grateful to S Haesen for bringing to our attention important references and relevant results, and to M S\'anchez for some comments. JMMS is supported by grants FIS2004-01626 (MICINN) and GIU06/37 (UPV/EHU).

\end{document}